\documentclass[11pt]{amsart} \usepackage{geometry} 
\usepackage[utf8x]{inputenc}
\usepackage{graphicx,color,amssymb,epstopdf} \newtheorem{tm}{Theorem}
\newtheorem{co}{Conjecture} \newtheorem{lm}[tm]{Lemma}
\theoremstyle{definition} 
\newtheorem{rmk}[tm]{Remark} \newtheorem{ex}[tm]{Example}
\newcommand{\bC}{\ensuremath{\mathbb C}}

\title[Powers of generic ideals and the WLP for powers of some
monomial CI's]{Powers of generic ideals and the weak Lefschetz
  property for powers of some monomial complete intersections}
\author{Mats Boij}\address{Department of Mathematics, KTH - Royal
  Institute of Technology}\email{boij@kth.se} \author{Ralf
  Fr\"oberg}\address{Department of Mathematics, Stockholm
  University}\email{ralff@math.su.se} \author{Samuel
  Lundqvist}\address{Department of Mathematics, Stockholm
  University}\email{samuel@math.su.se} \date{}
  
  \thanks{The first author was partially supported by the grant VR2013-4545.}

\begin{document}
\begin{abstract}
  Given an ideal $I=(f_1,\ldots,f_r)$ in $\mathbb C[x_1,\ldots,x_n]$
  generated by forms of degree $d$, and an integer $k>1$, how large
  can the ideal $I^k$ be, i.e., how small can the Hilbert function of
  $\mathbb C[x_1,\ldots,x_n]/I^k$ be? If $r\le n$ the smallest Hilbert
  function is achieved by any complete intersection, but for $r>n$,
  the question is in general very hard to answer.  We study the
  problem for $r=n+1$, where the result is known for $k=1$. We also
  study a closely related problem, the Weak Lefschetz property, for
  $S/I^k$, where $I$ is the ideal generated by the $d$'th powers of
  the variables.
\end{abstract}

\maketitle

\section{Introduction}
Let $I=(f_1,\ldots,f_r)$ be an ideal generated by forms of degree $d$
in $\bC[x_1,\ldots,x_n]$, and let $ k\ge1$. How large can the ideal
$I^k$ be, i.e., how small can the Hilbert function of
$\bC[x_1,\ldots,x_n]/I^k$ be? It is clear that we get the smallest
Hilbert series if the $f_i$'s are general. If $r\le n$, $I$ is a
complete intersection, and the Hilbert series (and even the graded
Betti numbers) for $I^k$ are known (\cite{GuardoVanTuyl}). But if $r>n$ not
much is known even if $k=1$, and we are not aware of any result for
$k>1$. For $k=1$ the main classes for which the result is known is
when $n\le3$ (\cite{Anick},\cite{Froberg}) or when $r=n+1$ (\cite{Stanley}).  In
all known cases the series for an ideal with $r$ general generators of
degree $d \geq 2$ in $n$ variables is $[(1-t^d)^r/(1-t)^n]$, where
$[\sum_{i\ge0}a_it^i]$ means truncate before the first nonpositive
term.  A first guess for the Hilbert series of
$\bC[x_1,\ldots,x_n]/I^k$, $I$ generated by $r$ general forms of
degree $d$, could be that the series for $k \gg1$ and $r>n$ is the
same as for ${k+r-1\choose r-1}$ general forms of degree $dk$.  We
will show that this is not always the case for geometric reasons.

We also study a closely related problem, the {\em Weak Lefschetz
  property} (WLP) for powers of some monomial complete intersections
ideals. Recall that a graded algebra $A$ satisfies the WLP if there
exists a linear form $L$ such that the multiplication map
$\times L\colon A_i\rightarrow A_{i+1}$ has maximal rank for all
degrees $i$. See \cite{MiglioreNagel} for a survey on the WLP. 

For a graded algebra $R = \oplus R_i$ we let $R(t)$ denote the Hilbert
series of $R$.

\section{Powers of ideals of general forms}
Let $f_1,f_2,\ldots,f_{n+1}$ be general forms of degree $d$ in
$\bC[x_1,\ldots,x_n]$.  We will start by showing that when
$k = d^{n-1}$, the dimension of
$\bC[x_1,\ldots,x_n]/(f_1,f_2,\ldots,f_{n+1})^k$ in degree $dk$ is one
less than expected.

 \begin{lm}\label{lm:unique} 
   For general forms $f_1,f_2,\dots,f_{n+1}$ of degree $d$ in
   $\bC[x_1,\ldots,x_n]$ there is exactly one relation of degree
   $d^{n-1}$ in $\bC[f_1,\ldots, f_{n+1}]$.
 \end{lm}
 
\begin{proof}
  General forms $f_1,f_2,\dots, f_{n+1}$ give a well-defined map
  $\Phi\colon\mathbb P^{n-1}\longrightarrow \mathbb P^{n}$ since they
  have no common zeroes. The closure of the image is a variety of
  dimension $n-1$, so it is a hypersurface. The inverse image of a
  general line gives a reduced complete intersection, as we can see
  from producing one such example by choosing the forms to be powers
  of general linear forms and by considering a line that gives the
  inverse image given by the radical complete intersection ideal
  $(f_1-f_{n},f_2-f_{n},\dots,f_{n-1}-f_n)$. Thus a general line meets
  the hypersurface in $d^{n-1}$ distinct points and hence the image is
  not contained in any hypersurface of lower degree than $d^{n-1}$. We
  conclude that there is a unique equation of degree $d^{n-1}$
  defining the image of $\Phi$ .
\end{proof}

In the next two subsections, we will focus on the situation in two and
three variables. We will need the following lemma.

\begin{lm}\label{zerodk1}
  Let $I$ be an ideal in $\bC[x_1,\ldots,x_n]$ generated by
  homogeneous elements in degree $d$ and suppose that $x_i^d \in I$
  for $i = 1, \ldots, n$.  Suppose that $\bC[x_1,\ldots,x_n]/I^k$ is
  zero in degree $a$ and that $a \geq d(n-1)$. Then
  $\bC[x_1,\ldots,x_n]/I^{k+1}$ is zero in degree $a+d$.
\end{lm}

\begin{proof}
  Let $R = \bC[x_1,\ldots,x_n]$ and write
  $R = R_0 \oplus R_1 \oplus \cdots$. By the assumption,
  $R_a \subset I^k$, so $R_{a} x_i^d \subset I^{k+1}$ for
  $i = 1, \ldots, n$. Let $m$ be a monomial in $R_{a+d}$. Since
  $a +d \geq dn$, there is at least one variable, say $x_j$, whose
  exponent in $m$ is greater than or equal to $d$. But then
  $m \in R_a x_j^d$, that is, $m \in I^{k+1}$.
\end{proof}

\subsection{The case $n=2$, $r=3$}
Let $R_{2,d,k}=\bC[x,y]/(f_1,f_2,f_3)^k$, where the $f_i$'s are
general forms of degree $d$. If $d=2$, then $(f_1,f_2,f_3)=(x,y)^2$,
so $R_{2,2,k}=\bC[x,y]/(x,y)^{2k}$. We assume in this section that
$d\ge3$.

\begin{lm}\label{le}
  If $k \geq d-2$, then
  $R_{2,d,k}(t)\ge\sum_{i=0}^{dk-1}(i+1)t^i+{d-1\choose2}t^{dk}$
  coefficientwise.
\end{lm}

\begin{proof}
  Consider the case when $k = d$. By Lemma~\ref{lm:unique} the
  dimension of $R_{2,d,k}$ in degree $dk$ is
  $d^2+1 - \left(\binom{d+2}{2}-1 \right) = \binom{d-1}{2}$.  This
  implies that in degree $k \geq d$, there are ${k-d+2\choose 2}$
  relations among the monomials of degree $k$ in the $f_i$'s. We have
  ${k+2\choose2}$ monomials of degree $k$ in the $f_i$'s.  Thus the
  dimension of $R_{2,d,k}$ in degree $dk$ is at least
  $dk+1-\left({k+2\choose2}-{k-d+2\choose2}\right)={d-1\choose2}$.

  Finally, let us consider the case when $k=d-1$ and $k = d-2$. The
  dimension of $R_{2,d,k}$ in degree $dk$ is at least
  $dk+1-{k+2\choose2}$. If $k = d-1$ or $k=d-2$, this equals
  $\binom{d-1}{2}$.
\end{proof}

\begin{lm}\label{gedk}
  If k $\geq d-2$, the Hilbert function of
  $\bC[x,y]/(x^d,y^d,x^{d-1}y)^k$ in degree $d k$ is ${d-1\choose2}$.
\end{lm}

\begin{proof}
  By Lemma \ref{le}, we have
  $\bC[x,y]/(x^d,y^d,x^{d-1}y)^k(t) \geq
  \sum_{i=0}^{dk-1}(i+1)t^i+{d-1\choose2}t^{dk}$.
  Thus it is enough to show that the Hilbert function of
  $\bC[x,y]/(x^d,y^d,x^{d-1}y)^k$ in degree $d k$ is at most
  ${d-1\choose2}.$

  Let $h=\min(k,d-1)$.  We construct $h+1$ separate groups indexed by
  $0,\ldots,h$, where we let the $b$'th group consist of the
  polynomials $x^{da} (x y^{d-1})^b y^{dc}$ of degree $dk$ such that
  $a+b \leq k$.  The leading monomials in this group are
  $$
  x^b y^{dk-b}, x^{b+d} y^{dk-(b+d)}, \ldots, x^{b+(k-b)d} y^{dk -
    (b+(k-b)d)},
  $$
  so there are $k-b+1$ leading monomials in group $b$.

  In total this gives
  $\sum_{b=0}^{h} (k-b+1) = (h+1)(k+1) - \binom{h+1}{2}$ elements.  If
  $k = d-2$, then the dimension in degree $dk$ is at most
  $d(d-2) +1 - (d-1)(d-1) + \binom{d-1}{2} = \binom{d-1}{2}.$ If
  $k = d-1$, then the dimension in degree $dk$ is at most
  $d(d-1) +1 - d^2 + \binom{d}{2} = \binom{d-1}{2}$. Finally, if
  $k \geq d$, the dimension in degree $dk$ is at most
  $kd + 1 - d(k+1)+\binom{d}{2} = \binom{d-1}{2}.$
\end{proof}

\begin{lm}\label{gedk1}
  The Hilbert function of
  the algebra $\bC[x,y]/(x^d,y^d,x^{d-1}y + xy^{d-1})^k$ is zero in
  degree $d k +1$ if $k \geq \max(d-3,1)$.
\end{lm}

\begin{proof}
  We will prove by induction that all the monomials of degree $dk+1$
  are in the ideal $J = (x^d,y^d,x^{d-1}y + xy^{d-1})^k$. Assume by
  induction on $\ell$ that we have all the monomials of the form
  $x^{id+j}y^{(k-i)d+1-j}$ and   $x^{(k-i)d+1-j}y^{id+j}$ in $J$ for $0\le
  i \le k-j$, $0 \le j \le \ell+1$. Denote the ideal generated by
  these monomials by $J^{(\ell)}$. Observe that $J^{(0)}
  =(x,y)(x^d,y^d)^k \subseteq J$ by the  defintion of $J$. 
  
  We compute for $0\le i \le k-\ell-1$ the following element in $J$
  \begin{multline*}
    x\cdot (x^{d-1}y+xy^{d-1})^{\ell+1} \cdot x^{id}y^{(k-i-\ell-1)d} 
    =\sum_{j=0}^{\ell+1} \binom{\ell+1}{j}
    x^{(i+j)d+\ell+2-2j}y^{(k-i-j)d-\ell-1+2j}  \\ 
    \equiv
    x^{id+\ell+2}y^{(k-i)d-\ell-1} \pmod{J^{(\ell)}}
  \end{multline*}
  since all but the first term are in $J^{(\ell)}$. By symmetry we
  also get that the following is in $J$
  \begin{multline*}
    y\cdot (x^{d-1}y+xy^{d-1})^{\ell+1} \cdot x^{(k-i-\ell-1)d} y^{id}
    =\sum_{j=0}^{\ell+1} \binom{\ell+1}{j}
    x^{(k-i-j)d-\ell-1+2j} y^{(i+j)d+\ell+2-2j}  \\ 
    \equiv
    x^{(k-i)d-\ell-1} y^{id+\ell+2} \pmod{J^{(\ell)}}
  \end{multline*}
  since all but the last term are in $J^{(\ell)}$.  Hence we have
  proved that $J^{(\ell+1)}\subseteq J$ for $0\le \ell\le k$.

Now we need to prove that $J^{(k)}$ contains all monomials of degree
$dk+1$. Any monomial of degree $dk+1$ can be written as $x^{jd+r}y^{(k-j)d-(r-1)}$,
where $0\le r<d$. This monomial is of the form
$x^{id+\ell+2}y^{(k-i)d-\ell-1}$ in $J^{(k)}$ if $r = \ell+2$ and $j\le k-\ell-1 = k-r+1$,
i.e., $r\le k-j+1$ and it is of the form
$x^{(k-i)d-\ell-1}y^{id+\ell+2}$ in $J^{(k)}$ if $jd+r=(k-i)d-\ell-1$, i.e.,
$j=k-i-1$, $r=d-\ell-1$ and $0\le i\le k-\ell-1$, i.e., $0\le k-j-1\le
r+k-d$, which can be written as $r\ge d-j-1\ge d-k$. We will get all
monomials in degree $dk+1$ if $d-j-1\le (k-j+1)+1$, i.e., when $d-3\le k$.
\end{proof}

\begin{tm}\label{main}
  If $k \geq d-2$, then
  \[
  R_{2,d,k}(t)=\sum_{i=0}^{dk-1}(i+1)t^i+{d-1\choose2}t^{dk}.
  \]
  The Betti numbers of $R_{2,d,k}$ are $\beta_{0,0}=1$,
  $\beta_{1,dk}=\frac{2dk-d^2+3d}{2}$, $\beta_{2,dk+1}=(dk-d^2+3d-2)$,
  $\beta_{2,dk+2}={d-1\choose2}$.
\end{tm}

\begin{proof}
  By Lemma \ref{gedk} and Lemma \ref{gedk1},
  $R_{2,d,k}(t)\le\sum_{i=0}^{dk-1}(i+1)t^i+{d-1\choose2}t^{dk}.$
  Together with Lemma~\ref{le} we get the claimed Hilbert series.
  
  The Hilbert series equals
  $(\sum_{i=0}^{dk+2}(-1)^i\beta_{i,j}t^{i+j})/(1-t)^2$. We know that
  the number of generators of $(f_1,f_2,f_3)^k$ are
  ${k+2\choose2}-{k-d+2\choose2}=\frac{2dk-d^2+3d}{2}$, and they are
  all of degree $dk$, so
  $\beta_1=\beta_{1,dk}=\frac{2dk-d^2+3d}{2}$. Also
  $\beta_{2,dk+2}={d-1\choose2}$ and $\beta_{2,i}=0$ if $i>dk+2$ since
  $R_{2,d,k}$ has socle of dimension ${d-1\choose2}$ in degree $dk$,
  and no socle in higher degree.
\end{proof}

 \begin{rmk}
   We are convinced that also the algebra
   $\bC[x,y]/(x^d,y^d,(x+y)^d)^k$, is zero in degree
   $d k +1$ for any $d$ and $k \geq \max(d-3,1)$, but we have not been
   able to  prove this. See also Conjecture \ref{co:dk1}.
 \end{rmk}

 \subsection{The case $n=3$, $r=4$}

 \begin{tm}\label{main2}
   The Hilbert series of $R_{3,2,k}$ is
   \[
   \sum_{i=0}^{2k-1}{i+2\choose2}t^i+(3k-1)t^{2k}.
   \]
 \end{tm}

 \begin{proof}
   A similar calculation as in the proof of Lemma \ref{le} shows that
   $R_{3,2,k}$ has at least dimension
   ${2k+2\choose2}-{k+3\choose3}+{k-4+3\choose3}=3k-1$ in degree
   $2k$. Thus we have an inequality. That there is equality follows
   from the following lemma.
 \end{proof}

 \begin{lm} \label{lemma:32k} The Hilbert series of
   \[
   \bC[x,y,z]/(x^2,y^2,z^2,(x+y+z)^2)^k=\bC[x,y,z]/(x^2,y^2,z^2,xy+xz+yz)^k
   \]
   is
   \[
   \sum_{i=0}^{2k-1}{i+2\choose2}t^i+(3k-1)t^{2k}.
   \]
 \end{lm}

\begin{proof}
  It is enough to show that the Hilbert series has at most dimension
  $3k-1$ in degree $2k$ and that the algebra is zero in degree $2k+1$.

  We first consider degree $2k$. We claim that all monomials of degree
  $2k$ except the $3k-1$ monomials
  \[x^{2k-1}z,x^{2k-3}z^3,\ldots,xz^{2k-1},
  y^{2k-1}z,y^{2k-3}z^3,\ldots,yz^{2k-1},
  xy^{2k-2}z,xy^{2k-4}z^3,\ldots,xy^2z^{2k-3}\]
  occur as leading monomials in lexicographic ordering.  We have
  $(x^2,y^2,z^2,(x+y+z)^2)=(x^2,y^2,z^2,f)$, where $f=xy+xz+yz$. If
  $m=x^ay^bz^c$ is a monomial of degree $2k$, we let
  $t(m)=(a \pmod 2,b \pmod 2,c \pmod 2)$. If $t(m)=(0,0,0)$, then $m$
  lies in the ideal. If $t(m)=(1,1,0)$, then $m=(xy)^{\min(a,b)}M$,
  where $t(M)=(0,0,0)$, so $m$ is the leading monomial of
  $f^{\min(a,b)}M$. If $t(m)=(1,0,1)$ we claim that $m$ is a leading
  monomial except when $b=0$ or $a=1$. If $t(m)=(0,1,1)$ we claim that
  $m$ is a leading monomial except when $a=0$. To prove the claims, it
  suffices to show that all monomials of the form $x^2yzM$,
  $t(M)=(0,0,0)$ and all monomials of the form $x^3y^2zM$,
  $t(M)=(0,0,0)$ are leading. Now $x^2yz$ is the leading monomial of
  $f^2-x^2y^2$, and $x^3y^2z$ is the leading monomial of
  $f(f^2-x^2y^2-x^2z^2)$.

  We now consider degree $2k+1$. A calculation shows that
  $\bC[x,y,z]/(x^2,y^2,z^2,xy+xz+yz)$ is zero in degree three and that
  $\bC[x,y,z]/(x^2,y^2,z^2,xy+xz+yz)^2$ is zero in degree five. Since
  $2 k + 1 \geq 2 (n-1)$ when $k \geq 2$, the remaining cases follow
  from Lemma \ref{zerodk1}.
\end{proof}
 
\begin{tm} \label{thm:33k} The Hilbert series of $R_{3,3,k}$ is
  $\sum_{i=0}^{3k-1}{i+2\choose2}t^i+(27k-56)t^{3k}$ when
  $9\le k \le 40$ and $[(1-t^{3k})^{k+3\choose3}/(1-t)^3]$ when $k<9$.
\end{tm}
\begin{proof}
 
  Consider first the case $k \geq 9$. By Lemma \ref{lm:unique} and a
  similar calculation as in the proof of Lemma \ref{le}, we get that
  $R_{3,3,k}$ has at least dimension
  ${3k+2\choose2}-{k+3\choose3}+{k-6\choose3}=27k-56$ in degree $3k$.
   
  Next, consider the case $k < 9$. If the $\binom{k+3}{3}$ generators
  were generic, the series would be
  $[(1-t^{3k})^{k+3\choose3}/(1-t)^3]$, so the series is a lower
  bound.
   
  To get an upper bound, it is enough to find an example.
  According to Macaulay2, the ideal
  $(x^{3},y^{3},z^{3},x^{2} y+11 x y^{2}-50 x^{2} z+48 x y z-29y^{2}
  z-9 x z^{2}+30 y z^{2}) \subset
  \mathbb{Z}/101\mathbb{Z}[x_1,x_2,x_3]$
  has the desired property for $k \leq 40$.
\end{proof}

 \begin{rmk}
   We are convinced that
   $R_{3,3,k}(t)=\sum_{i=0}^{3k-1}{i+2\choose2}t^i+(27k-56)t^{3k}$ in
   general, cf. Conjecture \ref{co:dk1}.
 \end{rmk}

 \subsection{The general case}
 We believe that some of the results have generalizations to any
 $n$. By Lemma \ref{lm:unique}, there is unique relation among the
 ${d^{n-1}+n\choose n}$ generators of $I^k$, that is,
 ${d^n+n-1\choose n-1} \geq {d^{n-1}+n\choose n}-1$.  To show that
 this relation is not trivial, we need to show that there is room for
 ${d^{n-1}+n\choose n}$ generators of degree $d^n$ in
 $\bC[x_1,\ldots,x_n]$.  When $(d,n) = (2,2)$, we have five generators
 but only room for four.  When $(d,n) = (3,2)$, we have ten generators
 and room for ten. In all other cases, we have a strict inequality.

 \begin{lm} \label{lm:binineq} Let
   $d,n \geq 2, (d,n) \neq (2,2), (3,2)$. Then
   ${d^n+n-1\choose n-1}>{d^{n-1}+n\choose n}$.
 \end{lm}

 \begin{proof}
   We need to show that
   \[
   n(d^n + n-1) \cdots (d^n + 1) > (d^{n-1} + n ) \cdots (d^{n-1} +1).
   \]
   We have $(d^n + n-1) \cdots (d^n + 1) > (d^n+1)^{n-1}$ and
   $ (d^{n-1} + n ) \cdot (d^{n-1} + n-1)^{n-1}>(d^{n-1} + n ) \cdots
   (d^{n-1} +1).$ Thus we are done if we can show that
   \[
   n\cdot \left( \frac{d^n + 1}{d^{n-1}+n-1} \right)^{n-1} > d^{n-1} +
   n .
   \]

   Now
   $\frac{d^n + 1}{d^{n-1}+n-1} = \frac{d \cdot (d^{n-1} + n-1) + 1 -
     d n+d}{d^{n-1}+n-1} > d(1 - \frac{n}{d^{n-1}})$.
   We have
   $\left(1 - \frac{n}{d^{n-1}}\right)^{n-1} =\binom{n-1}{n-1} -
   \binom{n-1}{n-1-1} \frac{n}{d^{n-1}} + \binom{n-1}{n-1-2}
   \frac{n^2}{(d^{n-1})^2} - \binom{n-1}{n-1-3}
   \frac{n^3}{(d^{n-1})^3} + \cdots$
   and $ \binom{n-1}{n-1-a} \frac{n^a}{(d^{n-1})^a} -$

   $ \binom{n-1}{n-1-(a+1)} \frac{n^{a+1}}{(d^{n-1})^{a+1}} =
   \binom{n-1}{n-1-a} \frac{n^a}{(d^{n-1})^a} \left(1-
     \frac{n-a-1}{a+1} \cdot \frac{n}{d^{n-1}} \right) > $
   $ \binom{n-1}{n-1-a} \frac{n^a}{(d^{n-1})^a}\left(1-
     \frac{n^2}{d^{n-1}} \right).$

   The inequality
   $1- \frac{n^2}{d^{n-1}} \geq 0 \Leftrightarrow d^{n-1} \geq n^2$
   holds true for $d \geq 3, n \geq 3$ and for $d=2, n \geq 7$. In
   these cases, we get the inequality
   $\left(1 - \frac{n}{d^{n-1}}\right)^{n-1} > 1-
   \frac{n^2}{d^{n-1}}$, so it is then enough to show that
   \[
   n d^{n-1} (1- \frac{n^2}{d^{n-1}}) > d^{n-1} + n.
   \]

   We have
   $n d^{n-1} (1- \frac{n^2}{d^{n-1}}) > d^{n-1} + n \Leftrightarrow
   d^{n-1} > \frac{n^3 + n}{n-1}$.
   This inequality holds true for $d \geq 2, n \geq 7$ and
   $n \geq 3, d \geq 3$.

   We are left with a few special cases only. In the $n = 2$ case, we
   have to check the inequality
   $2 \cdot \frac{d^2 +1}{d+1} > d + 2 \Leftrightarrow d^2 > 3d$,
   which holds when $d \geq 4$.

   The remaining cases are covered by the conditions
   $3\leq n \leq 6 ,d = 2$ and then the inequality
   $\binom{d^n + n-1}{n-1} > \binom{d^{n-1} + n}{n}$ is checked to be
   true by hand.

 \end{proof}

 By Lemma \ref{lm:unique} and \ref{lm:binineq}, the dimension of
 $R_{n,d,k}$ in degree $dk$ is at least
 ${dk+n-1\choose n-1}-{k+n\choose n}+{k-d^{n-1}+n\choose n}$ when
 $k \geq d^{n-1}$, with equality when $k = d^{n-1}$. We conjecture
 that the algebra is zero in higher degrees.

 \begin{co} \label{co:dk1} If $k\ge d^{n-1}$ and $(d,n) \neq (2,2)$,
   the Hilbert series of the algebra
   $\bC[x_1,\ldots,x_n]/(f_1,\ldots,f_{n+1})^{k}$, $f_i$ general of
   degree $d$, is
   \[
   \sum_{i=0}^{dk-1}{i+n-1\choose n-1}t^i+ \left({dk+n-1\choose
     n-1}-{k+n\choose n}+{k-d^{n-1}+n\choose n}\right)t^{dk}.
   \]
 \end{co}

\begin{rmk}
  Theorem \ref{main} shows that Conjecture \ref{co:dk1} holds in the the
  case $n = 2$. Theorem \ref{main2} shows that
  Conjecture \ref{co:dk1} holds in the the case $n = 3$, $d =2$, while
  Theorem \ref{thm:33k} shows that the conjecture holds when
  $n = 3, d = 3, k \leq 40$.
\end{rmk}
  
We now look at the case when $k<d^{n-1}$. There are obvious relations between the
generators of $(f_1,\ldots,f_{n+1})^k$ of type
$f_1\cdot f_2^k=f_2\cdot(f_1f_2^{k-1})$. These are of degree
$d(k+1)$. Now, if the algebra
$\bC[x_1,\ldots,x_n]/(f_1,\ldots,f_{n+1})^k$ is zero in degree $d(k+1)$,
these relations do not show up. This happens if
\begin{equation*}
  \begin{split}
    S_{n,d,k} & ={d+n-1\choose n-1}{k+n\choose n}-\\
    & \left({k+n\choose n}(n+1)-{k+1+n\choose
        n}\right)-{d(k+1)+n-1\choose n-1}\ge0.
  \end{split}
\end{equation*}

\begin{co} \label{co:gen} Let $k<d^{n-1}$ and suppose that for
  $R_{d,k,n}=\bC[x_1,\ldots,x_n]/(f_1,\ldots,f_{n+1})^k$ we have
  $S_{n,d,k}\ge0$. Then the Hilbert series of $R_{d,k,n}$ equals the
  Hilbert series of $\bC[x_1,\ldots,x_n]/I$, where $I$ is an ideal
  generated by $\binom{k + n}{n}$ general forms of degree $d k$.
\end{co}

It is known that $\bC[x_1,\ldots,x_n]/(f_1,\ldots,f_{n+1})$ has the
same Hilbert series whether the $f_i$'s are general forms of degree
$d$, or $d$'th powers of general linear forms \cite{Stanley}. In the second
case we can assume that
$(f_1,\ldots,f_{n+1})=(x_1^d,\ldots,x_n^d,(x_1+\cdots+x_n)^d)$.  We
will end this section by giving an explicit relation of degree $d^n$
in $\bC[x_1^d,\ldots,x_n^d, (x_1+\cdots+x_n)^d]$.

\begin{tm} \label{thm:conjugate} In
  $T=\bC[x_1,\ldots,x_{n+1}]/(x_1^d,\ldots,x_{n+1}^d)^{d^{n-1}}$ we
  have
  \[F=\prod(x_1+\epsilon_1x_2+\cdots+\epsilon_{n-1}x_{n}+\epsilon_nx_{n+1})=0,\]
  where the $\epsilon_i$'s vary over all $d^{n}$ combinations of a
  $d$'th root of unity. In particular, $F/(x_1+x_2+\cdots+x_{n+1})$ is
  a form of degree $d^{n}-1$ which is in the kernel of the map
  $\times (x_1+x_2+\cdots+x_{n+1})\colon T\rightarrow T$.
\end{tm}

\begin{proof}
  Let $G_d\subseteq \mathbb C^\ast$ be the group of $d$'th roots of
  unity. The form $F$ is invariant under the action of $G_d^{n+1}$,
  acting with multiplication on the variables, so it has to be a
  polynomial in $x_1^d,x_2^d,\dots,x_{n+1}^d$, which means that $F=0$
  in $T$.
\end{proof}
  
Let us now consider the form $F$ as an element in
$\bC[x_1,\ldots,x_{n+1}]$.  When $n+1 = 2$ and $d = 2$,
$F = (x_1+x_2)(x_1-x_2) = x_1^2 - x_2^2$ --- the conjugate rule, and
for general $d$, $F=x_1^d-(-1)^{d}x_2^d$.

When $n+1 \geq 3$, the form $F$ is symmetric. To show this, notice
first the the form is invariant under permutation of the variables
$x_2,x_3,\ldots,x_{n+1}$. Thus it is enough to show that $F$ is
invariant with respect to the transposition $(12)$.  Let $G$ denote
the result after letting $(12)$ act on $F$. Let $\epsilon$ be a
primitive $d$'th root of unity and let
$f_i = \prod(x_1+\epsilon^i x_2+\epsilon_{2}x_3
+\cdots+\epsilon_{n-1}x_{n}+\epsilon_nx_{n+1})$,
so that $F = \prod_i f_i.$ Let
$g_i = \prod(\epsilon_1^i x_1+x_2+\epsilon_{2}x_3 +
\cdots+\epsilon_{n-1}x_{n}+\epsilon_nx_{n+1}).$
Since $n+1\geq 3$, we have
$f_i = \prod \epsilon^{d-i} (x_1+\epsilon^i x_2+\epsilon_{2}x_3
+\cdots+\epsilon_{n-1}x_{n}+\epsilon_nx_{n+1}) =$
$g_{d-i}$, so $F = \prod_i f_i = \prod_i g_i = G$.

\begin{ex}
  In general, it seems hard to get an explicit description of $F$. We
  can get it for $d=2$, $n+1=3,4,5$ and for $n+1=d=3$. It is the
  symmetrization of the following polynomials.

\[\begin{array}{l|l}
    (n+1,d)\\ \hline \\
    (3,2)&x^4+x^2y^2\\
    \\
    (4,2)& x^8-4x^6y^2+6x^4y^4+4x^4y^2z^2\\ 
    \\
    (5,2)&x^{16}-8x^{14}y^2+28x^{12}y^4+40x^{12}y^2z^2-56x^{10}y^6-72x^{10}y^4z^2 \\
         &       -176x^{10}y^2z^2u^2 +70x^8y^8+40x^8y^6z^2+36x^8y^4z^4+344x^8y^4z^2u^2\\
         &      -757x^8y^2z^2u^2v^2+16x^6y^6z^4-416x^6y^6z^2u^2-272x^6y^4z^4u^2 \\
         &                +928x^6y^4z^2u^2v^2+2008x^4y^4z^4u^4-1520x^4y^4z^4u^2v^2\\
    \\
    (3,3)&x^{27}+36x^{24}y^3-9x^{21}y^3z^3+684x^{18}y^9-234x^{18}y^6z^3+3339x^{18}y^3z^3u^3
    \\&+126x^{15}y^{12}-711x^{15}y^9z^3+513x^{15}y^6z^6+1512x^{15}y^6z^3u^3-990x^{12}y^{12}z^3
    \\&+ 2961x^{12}y^9z^6-12222x^{12}y^6z^6u^3+278371x^{12}y^6z^6u^3-12171x^9y^9z^9
    \\&-6867x^9y^9z^6u^3+120312x^9y^6z^6u^6
    \\
  \end{array}
  \]
\end{ex}

\section{The Weak Lefschetz property}

We now turn to the WLP for
$T_{n,d,k} := \mathbb{C}[x_1,\ldots,x_n]/(x_1^d,\ldots,x_n^d)^k.$ The algebra $T_{2,d,k}$ has the WLP, since quotients of $\mathbb{C}[x_1,x_2]$ has the Strong Lefschetz property (\cite{Harimaetal}), which implies the WLP. 
When $k=1$, the algebra is of the form
$\mathbb{C}[x_1,\ldots,x_n]/(x_1^d,\ldots,x_n^d)$, which is a monomial
complete intersection, and therefore also $T_{n,d,1}$ has the Strong Lefschetz property (\cite{Stanley}).  When $d = 1$, we have
$T_{n,1,k} = \bC[x_1,\ldots,x_n]/(x_1,\ldots,x_n)^k$. Since
$\bC[x_1,\ldots,x_n]$ has the WLP, so has
$\bC[x_1,\ldots,x_n]/(x_1,\ldots,x_n)^k$.  When $n \geq 3$ and $d,k \geq 2$, the
situation is more involved. 

The WLP for algebras given by monomial ideals have been
studied before, for example in \cite{CookIINagel1} and
\cite{CookIINagel2} but not for our kind of monomial ideal.

Since $T_{n,d,k}$ is a monomial algebra, $(\mathbb C^\ast)^n$ acts on
$T_{n,d,k}$ and therefore any general linear form can be identified
with $L=x_1+ \cdots + x_n$. When taking the quotient by the ideal
$(L)$, we get
$\tilde
T_{n,d,k}:=\bC[x_1,\ldots,x_{n-1}]/(x_1^d,\ldots,x_{n-1}^d,(x_1+\cdots
+ x_{n-1})^d)^k$.
Thus $T_{n,d,k}$ has the WLP if and only if
$[(1-t)T_{n,d,k}(t)]=\tilde T_{n,d,k}(t)$.

From Lemma \ref{lm:binineq} and Theorem \ref{thm:conjugate} we can
immediately get a negative result on the WLP.

\begin{tm} \label{tm:notwlpunique} Suppose that
  $k \geq d^{n-2}, n \geq 3, (n,d) \neq (3,2)$. Then $T_{n,d,k}$ fails
  the WLP.
\end{tm}
\begin{proof}
  By Lemma \ref{lm:binineq}, $\tilde T_{n,d,k}$ is non-zero in degree
  $dk$.  Thus it is enough to show that the map
  $\times (x_1+x_2+\cdots + x_n) \colon (T_{n,d,k})_{dk-1} \rightarrow
  (T_{n,d,k})_{dk}$ is not injective.
  
  Suppose that $k = d^{n-2}.$ By Theorem \ref{thm:conjugate}, the form
  $(\prod (x_1 + \epsilon_1 x_2 + \cdots +\epsilon_{n-1}
  x_n))/(x_1+x_2+\cdots+x_n)$
  of degree $dk-1$ is in the kernel of the map. Hence it cannot be
  injective.
  
  Suppose instead that $k > d^{n-2}.$ Then
  $x_1^{k-d^{n-2}} (F/(x_1+x_2+\cdots+x_n))$ is in the kernel.
\end{proof}

\subsection{The case $n=3$} 
We have a conjecture on the WLP for $T_{3,d,k}$.

\begin{co} \label{co:conjwlp3} The algebra
  $\bC[x,y,z]/(x^d,y^d,z^d)^k$ has the WLP if and only if one of the
  following conditions is satisfied.
  \begin{enumerate}
  \item $d \leq 2$,
  \item $k \leq 2$,
  \item $d> k = 2j+1 \in \{3,7\}$, $d \neq (j+2)(2 \ell+1),$
  \item $d> k =2j+1 >2, k \notin \{3,7\}$,
    $d \notin \{ (j+2)(2 \ell+1)-1, (j+2)(2 \ell+1), (j+2)(2 \ell+1)+1
    \},$
  \item $d> k = 2j >2$,
    $d \notin \{ (j+1)(2 \ell+1)+\ell, (j+1)(2 \ell+1)+\ell+1 \}.$
  \end{enumerate}
\end{co}

We have already proven that the algebra $T_{3,d,k}$ fails the WLP when
$3 \leq d \leq k$ and that $T_{3,d,1}$ and $T_{3,1,k}$ have the WLP.
We will now prove Conjecture \ref{co:conjwlp3} in some more cases.

\begin{tm}\label{WLPd2}
  The algebra $T_{3,2,k}$ has the WLP.
\end{tm}
\begin{proof}
  It is enough to show that $\tilde T_{3,2,k}$ is zero in degree $2k$,
  which follows from the observation that
  $(x^2,y^2,(x+y)^2)^k = (x^2,y^2,xy)^k$.
\end{proof}
In Theorem~\ref{tm:necessity} below we prove the necessity of
conditions (3)-(5) of Conjecture \ref{co:conjwlp3}.

\begin{tm}\label{tm:necessity}
  The algebra $T_{3,d,k}$ fails to have the WLP in the cases
  \begin{enumerate}
  \item $k = 2j+1>2$, $d = (j+2)(2 \ell+1),$ where $\ell \ge 1$.
  \item $k =2j+1 >2$, $k \notin \{3,7\}$, $d = (j+2)(2 \ell+1)\pm 1$,
    where $\ell \ge 1$.
  \item $k = 2j >2$,
    $d \in\{ (j+1)(2 \ell+1)+\ell, (j+1)(2 \ell+1)+\ell+1 \}.$
  \end{enumerate}
\end{tm}

\begin{proof}
  By \cite{GuardoVanTuyl} we have that the Hilbert series of $T_{3,d,k}$ is
  given by
$$    T_{3,d,k}(t)=\frac{1-{k+2\choose2}t^{dk}+(k^2+2k)t^{d(k+1)}-
  {k+1\choose2}t^{d(k+2)}}{(1-t)^3}.$$
From this, we get that in the range $dk\le i\le d(k+1)$ the Hilbert
function is given \[\binom{i+2}2-\binom{i-dk+2}2\binom{k+2}2\]
with first difference equal to $(i+1)-(i-dk+1)(k+1)(k+2)/2$ which is
positive when $i+1<d(k+1)(k+2)/(k+3)$ and negative when
$i+1>d(k+1)(k+2)/(k+3)$. The turning point is in the interval where
this expression for the Hilbert function is valid.
  
In order to show that the WLP fails, we will use the representation
theory of the symmetric group, $S_3$. Since the algebra is monomial,
it is sufficient to consider multiplication by the linear form
$L = x+y+z$ which is symmetric and hence gives an equivariant map. In
all of the cases we consider, we have that
$d = (2\ell+1)(k+3)/2 +\epsilon$, where $-1\le \epsilon\le 1$.  We
will look at the multiplicity of the trivial and the alternating
representations in $T_{3,d,k}$ in degrees $dk+2\ell-1$ and
$dk+2\ell$. We compute these multiplicities by means of the
characters. In degree $i$ of $T_{3,d,k}$ we have
$\binom{i+2}{2}-\binom{k+2}{2}\binom{i-dk+2}{2}$ monomials that are
all fixed by the identity permutation. The other two even permutations
fixes only powers of $xyz$ so the character is $1$ if
$i\equiv 0\pmod 3$ and $k\not\equiv 0\pmod 3$ and $0$ otherwise. The
transposition $(12)$ fixes the monomials of the form $(xy)^mz^n$ and
we have that the value of the character on the three transpositions is
\[
\left\lfloor\frac{i+2}{2}\right\rfloor
-\left\lfloor\frac{k+2}{2}\right\rfloor
\left\lfloor\frac{i-dk+2}{2}\right\rfloor.
\]
Thus we get that the multiplicity of the trivial representation is
\[
\frac16\left(\binom{i+2}{2}-\binom{k+2}{2}\binom{i-dk+2}{2} + m + 3
  \left(\left\lfloor\frac{i+2}{2}\right\rfloor
    -\left\lfloor\frac{k+2}{2}\right\rfloor
    \left\lfloor\frac{i-dk+2}{2}\right\rfloor\right)\right)
\]
where the middle term $m$ is $-2$, $0$ or $2$ depending on the
congruence modulo $3$ of the sum of the other terms. In the same way
we get the multiplicity of the alternating representation as
\[
\frac16\left(\binom{i+2}{2}-\binom{k+2}{2}\binom{i-dk+2}{2} + m - 3
  \left(\left\lfloor\frac{i+2}{2}\right\rfloor
    -\left\lfloor\frac{k+2}{2}\right\rfloor
    \left\lfloor\frac{i-dk+2}{2}\right\rfloor\right)\right)
\]
When we compute the difference of the Hilbert function between degree
$dk+2\ell-1$ and $dk+2\ell$ we get
\begin{multline*}
  \binom{dk+2\ell+2}2-
  \binom{dk+2\ell+1}2-\binom{k+2}{2}\left(\binom{2\ell+2}2-\binom{2\ell+1}2\right)
  \\= dk+2\ell +1 -\binom{k+2}2(2\ell+1) = dk-(2\ell+1)\frac{k(k+3)}2
  = k\epsilon.
\end{multline*}
The difference in the number of monomials fixed by a transposition is
\begin{multline*}
  \left\lfloor\frac{dk+2\ell+2}{2}\right\rfloor -
  \left\lfloor\frac{dk+2\ell+1}{2}\right\rfloor -
  \left\lfloor\frac{k+2}{2}\right\rfloor
  \left(\left\lfloor\frac{2\ell+2}{2}\right\rfloor
    -\left\lfloor\frac{2\ell+1}{2}\right\rfloor \right)
  \\
  = \left\lfloor\frac{dk}{2}\right\rfloor -
  \left\lfloor\frac{dk-1}{2}\right\rfloor
  -\left\lfloor\frac{k+2}{2}\right\rfloor =
  \begin{cases}-j-1 &\text{if $k$ and $d$ are odd,}\\-j&
    \text{otherwise.}
  \end{cases}
\end{multline*}
When $k$ is odd and $\epsilon=\pm1$, $d$ is odd when $j$ is even and
the difference above can be expressed as $-2\lfloor j/2\rfloor-1$.
When $k$ is odd and $\epsilon=0$, $d$ is odd when $j$ is odd and the
same difference can be expressed as $-2\lfloor(j+1)/2\rfloor$.

First we consider the case $k=2j$ even and $\epsilon=\pm1/2$. We
compute the difference in the multiplicity of the trivial
representation when $\epsilon>0$ and the alternating representation
when $\epsilon<0$. This difference equals
\[
\frac16\left(\pm j +m \pm 3(-j)\right) = \mp
\left\lfloor\frac{j+1}3\right\rfloor
\]
which has opposite sign to the difference in the Hilbert function when
$j+1\ge 3$. Thus the multiplication by $L=x+y+z$ cannot have maximal
rank by Schur's lemma.

For the case $k=2j+1$ and $\epsilon = \pm1$, we have the corresponding
difference
\[
\frac16\left(\pm (2j+1)+m\pm 3\left(-2\left\lfloor\frac
      j2\right\rfloor-1\right)\right) = \mp
\left(\left\lfloor\frac{j}2\right\rfloor-\left\lfloor\frac{j}3\right\rfloor\right).
\]
Again, this has different sign than the difference in the Hilbert
function when $\lfloor j/2\rfloor > \lfloor j/3\rfloor$, i.e., when
$j=2$ or $j\ge 4$.

In the last case $k=2k+1$ and $\epsilon=0$ it is sufficient to show
that the multiplicity of the trivial representation changes since the
Hilbert function has difference zero. The difference in the
multiplicity of the trivial representation is
\[
\frac16\left( 0 + m +3\left(-2\left\lfloor\frac
      {j+1}2\right\rfloor\right)\right) = - \left\lfloor\frac
  {j+1}2\right\rfloor
\]
which is negative for all $j\ge 1$.
\end{proof}

\begin{rmk}
  We can see that the argument of the proof does not work when
  $\epsilon =\pm 3/2$ or $\epsilon = \pm2$, so in these cases we have
  the same turning point for the Hilbert function of the three
  isotypic components.
\end{rmk}

\begin{rmk}
  By computations in Macaulay2, we have verified the unproven parts of
  Conjecture \ref{co:conjwlp3} for $d,k \leq 30$.
\end{rmk}

\subsection{The general case}

As the number of variables increases, the number of non-trivial pairs
$(d,k)$ for which $\mathbb{C}[x_1,\ldots,x_n]/(x_1^d, \ldots,x_n^d)^k$
has the WLP seems to decrease.

\begin{table}[h!]
$$
\begin{array}{l|l} 
  n&\text{WLP} \\
  \hline
  4 & (2,2), (2,3), (3,2), (4,2) \\
  5 & (2,2), \ldots,(2,7),(3,2),(4,2),(4,3),(6,2)\\
  6 & (2,2), (2,3), (3,2)\\
  7 & (2,2), (2,3), (3,2)\\ 
  8 & (2,2), (3,2) \\
  9 & (2,2), (2,3), (3,2) \\
  10 & (2,2), (3,2)  \\
  11 & (2,2), (2,3)  \\
  12, 14, 16 & (2,2) \\
  2a+1  & (2,2) \\
\end{array}
$$
\caption{The right hand column consists of pairs $(d,k)$ with $d,k \geq 2$  for which 
  $\mathbb{C}[x_1,\ldots,x_n]/(x_1^d, \ldots,x_n^d)^k$ has the WLP, detected by 
  computations in Macaulay2, except for
  the last row, which relies on Theorem \ref{tm:oddwlp}.
}
\label{table:wlp}
\end{table}

In Table \ref{table:wlp} we list the non-trivial cases where have been
able detect the WLP.  Based upon these observations together with some
negative results for the WLP in positive characteristic, we believe
that Table $\ref{table:wlp}$ is complete, except for the trivial
pairs, when $n \leq 10$.  When $n \geq 11$, we guess that the WLP
holds at most in the cases $(2,2), (2,3),(3,2)$ for $n$ odd, and at
most in the cases $(2,2), (3,2)$ for $n$ even.  We can support our
guesses with two theoretical results.  \vspace{1cm}

\begin{tm}
  \begin{itemize}
  \item The algebra $T_{4,2,k}$ has the WLP if and only if $k \leq 3$.
  \item The algebra $T_{4,3,k}$ has the WLP if and only if $k \leq 2$.
  \item The algebra $T_{4,4,k}$ has the WLP if and only if $k \leq 2$.
  \item The algebra $T_{5,2,k}$ has the WLP if and only if $k \leq 7$.
  \end{itemize}
\end{tm}

\begin{proof}
  By Theorem \ref{tm:notwlpunique}, the algebra $T_{4,2,k}$ fails the
  WLP when $k \geq 2^2= 4$, $T_{4,3,k}$ fails the WLP when
  $k \geq 3^2= 9$, $T_{4,4,k}$ fails the WLP when $k \geq 4^2= 16$,
  $T_{5,2,k}$ fails the WLP when $k \geq 2^3= 8$,

  Hence it is enough to check the WLP over $\mathbb{Q}$ for the
  remaining cases.  This has been done with Macaulay2.
\end{proof}

Finally, we can prove that
$\mathbb{C}[x_1,\ldots,x_n]/(x_1^2, \ldots,x_n^2)^2$ has the WLP when
$n$ is odd. We believe that it is also true for $n$ even, but our
method of proof does not apply in that case.

\begin{tm} \label{tm:oddwlp} For any odd $n$, the algebra
  $\mathbb{C}[x_1,\ldots,x_n]/(x_1^2,\ldots,x_n^2)^2$ has the WLP.
\end{tm}
\begin{proof}
  Let $R = \mathbb{C}[x_1,\ldots,x_n]/(x_1^2,\ldots,x_n^2)^2.$ When
  $d \geq 2$, write
  $R_d = A_{d,1} \oplus \cdots \oplus \cdots A_{d,n-1} \oplus B_d$,
  where $A_{d,i}$ is spanned by all monomials of the form $x_i^2 M$
  with $M$ a squarefree monomial, and where $B_d$ is the set of
  squarefree monomials of degree $d$ together with all monomials of
  the form $x_n^2 M$, with $M$ squarefree. We have
  $|A_{d,i}| = \binom{n}{d-2}$ and
  $|B_d| = \binom{n}{d} + \binom{n}{d-2}$, where $|V|$ denotes the
  dimension of the vector space $V$. Notice that
  $R_i A_{d,j} \subset A_{i+d,j}$.

  Multiplication by $(x_1+ \cdots + x_n)$ on $A_{d,i}$ agrees with
  multiplication by $(x_1+ \cdots + x_n)$ on the basis in degree $d-2$
  in $C[x_1,\ldots,x_n]/(x_1^2,\ldots,x_n^2)$. This algebra has the
  SLP, so multiplication by $(x_1+ \cdots + x_n)$ on $A_{d,i}$ has
  full rank.

  Since
  $(x_1^2,\ldots,x_n^2)^2 \subseteq (x_1^2,\ldots,x_{n-1}^2, x_n^4)$,
  the dimension of $(x_1+ \cdots + x_n) B_{d}$ in $R$ is greater than
  or equal to the dimension of $(x_1+ \cdots + x_n) B_{d}$ in the
  algebra
  $ C:= \mathbb{C}[x_1,\ldots,x_n]/(x_1^2,\ldots,x_{n-1}^2, x_n^4)$,
  where we abuse notation and regard $B_d$ as a part of $C.$ This
  algebra also has the SLP, so the dimension of
  $(x_1+ \cdots + x_n) B_{d}$ in $C$ equals
  $\min(|B_d|,|B_{d+1}|) = \min \left(\binom{n}{d} + \binom{n}{d-2},
    \binom{n}{d+1} + \binom{n}{d-1} \right)$.

  If $|B_d| \leq |B_{d+1}|$ and $|A_d| \leq |A_{d+1}|$ we can conclude
  that multiplication by $(x_1+\cdots +x_n)$ on $R_d$ is injective. If
  $|B_d| \geq |B_{d+1}|$ and $|A_{d,i}| \geq |A_{d+1,i}|$, we can
  conclude that multiplication by $(x_1+\cdots +x_n)$ on $R_d$ is
  surjective.

  Thus we are left with two cases.
  
  In the first case we have $|B_d| < |B_{d+1}|$ and
  $|A_{d,i}| > |A_{d+1,i}|$, that is, we have the inequalities
  $\binom{n}{d} + \binom{n}{d-2} < \binom{n}{d+1} + \binom{n}{d-1}$
  and $\binom{n}{d-2} > \binom{n}{d-1}$. But
  $\binom{n}{d-2} > \binom{n}{d-1}$ implies that
  $\binom{n}{d} > \binom{n}{d+1}$, so
  $\binom{n}{d} + \binom{n}{d-2} > \binom{n}{d+1} + \binom{n}{d-1}$.
  Thus the inequalities $|B_d| < |B_{d+1}|$ and
  $|A_{d,i}| > |A_{d+1,i}|$ cannot be simultaneously satisfied.

  In the second case we have $|B_d| > |B_{d+1}|$ and
  $|A_{d,i}| < |A_{d+1,i}|$. For $\binom{n}{d-2} < \binom{n}{d-1}$ to
  hold we need to have $d-1 \leq \lfloor n/2 \rfloor$. For
  $\binom{n}{d} + \binom{n}{d-2} > \binom{n}{d+1} + \binom{n}{d-1}$ to
  simultaneously hold, we especially need the inequality
  $\binom{n}{d} > \binom{n}{d+1}$ to hold. This inequality is
  satisfied if and only if $d \geq \lfloor (n+1)/2 \rfloor$.

  Now the inequalities $d-1 \leq \lfloor n/2 \rfloor$ and
  $d \geq \lfloor (n+1)/2 \rfloor$ together gives us that
  $d = \lfloor n/2 \rfloor + 1.$ In the case $n = 2k+1$, we have
  $|B_{k+1}| = \binom{2k+1}{k+1} = \binom{2k+1}{k+1} +
  \binom{2k+1}{k-1} = \binom{2k+1}{k} + \binom{2k+1}{k+2} =
  |B_{k+2}|$,
  so the inequality $|B_d| > |B_{d+1}|$ is not fulfilled. This shows
  that the case $|B_d| > |B_{d+1}|$ and $|A_{d,i}| < |A_{d+1,i}|$ is
  empty when $n$ is odd, that is, $R$ has the WLP when $n$ is odd.
  
  Remark: When $n = 2k$, we can however not draw any conclusion
  regarding the rank of the map $\times (x_1+ \cdots + x_n)$ from
  degree $k+1$ to $k+2$.

\end{proof}


\end{document}